\documentclass{amsart}
\usepackage{amsmath,amsthm,amssymb}
\usepackage{enumerate}
\usepackage[all]{xy}

\newcommand{\C}{\mathbb{C}}

\newcommand{\R}{\mathbb{R}}
\newcommand{\cov}{\nabla}
\newcommand{\curv}{\widetilde{R}}

\newcommand{\id}{\mathrm{Id}}
\DeclareMathOperator{\im}{\mathrm{Im}}
\newcommand{\lie}{\mathcal{L}}
\DeclareMathOperator{\rank}{\mathbf{rk}}

\DeclareMathOperator{\trace}{\mathrm{tr}}

\newtheorem{theorem}{Theorem}[section]
\newtheorem{proposition}[theorem]{Proposition}
\newtheorem{proposition-definition}[theorem]{Proposition-Definition}

\theoremstyle{remark}

\newtheorem{remark}[theorem]{Remark}

\theoremstyle{definition}
\newtheorem{definition}[theorem]{Definition}

\title{Nearly Sasakian manifolds revisited}
\author[B. Cappelletti-Montano]{Beniamino Cappelletti-Montano}
 \address{Dipartimento di Matematica e Informatica, Universit\`a degli Studi di
 Cagliari, Via Ospedale 72, 09124 Cagliari, Italy}
 \email{b.cappellettimontano@gmail.com}

\author[A. De Nicola]{Antonio De Nicola}
 \address{Dipartimento di Matematica, Universit\`a degli Studi di Salerno, Via Giovanni Paolo II 132, 84084 Fisciano, Italy}
 \email{antondenicola@gmail.com}

\author[G. Dileo]{Giulia Dileo}
 \address{Dipartimento di Matematica, Universit\`a degli Studi di
 Bari Aldo Moro, Via E. Orabona 4, 70125 Bari, Italy}
 \email{giulia.dileo@uniba.it}

\author[I. Yudin]{Ivan Yudin}
 \address{CMUC, Department of Mathematics, University of Coimbra, 3001-501 Coimbra, Portugal}
 \email{yudin@mat.uc.pt}


\thanks{This work was partially supported by 
Fondazione di Sardegna and Regione Autonoma della  Sardegna, Project GESTA and KASBA,
 by MICINN (Spain) grant MTM2015- 64166-C2-2-P,
by CMUC -- UID/MAT/00324/2013, funded by the Portuguese
 Government through FCT/MEC and co-funded by the European Regional Development
Fund through the Partnership Agreement PT2020
and by the exploratory research project in the frame of Programa Investigador FCT IF/00016/2013.}

\begin{document}
\maketitle

\begin{center}
\emph{Dedicated to Prof. David E. Blair on the occasion of his 78th birthday}
\end{center}

\begin{abstract}
We provide a new, self-contained and more conceptual proof of the result that an
almost contact metric manifold of dimension greater than $5$ is Sasakian if and
only if it is nearly Sasakian.  \end{abstract}

\section{Introduction}
A Sasakian manifold $M$ is a contact metric manifold that satisfies a normality
condition, encoding the integrability of a canonical almost complex structure on
the product $M\times \mathbb{R}$. Several equivalent characterizations of this
class of manifolds, in terms of Riemannian cone, or transversal structure, or
curvature, are also known. In particular one can show that an almost contact
metric structure $(g,\phi,\xi,\eta)$ is Sasakian if and only if the covariant
derivative of the endomorphism $\phi$ satisfies
\begin{equation}\label{sasakian-condition}
(\nabla_{X}\phi)Y-g(X,Y)\xi+\eta(Y)X=0,
\end{equation}
for all vector fields $X,Y\in\Gamma(TM)$. A relaxation of this notion was
introduced by Blair, Showers and Yano in \cite{blair_nearly_sasakian}, under the
name of nearly Sasakian manifolds, by requiring that just the symmetric part of
\eqref{sasakian-condition} vanishes. Later on, several important properties of
nearly Sasakian manifolds were discovered by Olszak (\cite{olszak}).  Nearly
Sasakian manifolds may be considered as an odd-dimensional analogue of
nearly K\"{a}hler manifolds. In fact, the prototypical example of nearly
Sasakian manifold is the $5$-sphere as totally umbilical hypersurface of
$\mathbb{S}^6$, endowed with the almost contact metric structure induced by the
well-known nearly K\"{a}hler structure of $\mathbb{S}^6$. Nevertheless, in
recent years several differences between nearly Sasakian and nearly K\"{a}hler
geometry were pointed out.  In particular, in \cite{mg} it was proved that the
$1$-form $\eta$ of any  nearly Sasakian manifold is necessarily a  contact form,
while the fundamental $2$-form of a nearly K\"{a}hler manifold is never
symplectic unless the manifold is K\"{a}hler. 
 A peculiarity of nearly Sasakian five dimensional manifolds, which are not Sasakian, is
 that upon rescaling the metric one can define a  Sasaki-Einstein  structure on them.
In fact one has an SU(2)-reduction of the  frame bundle.
 Conversely, starting  with a five dimensional manifold with a Sasaki-Einstein
 SU(2)-structure it is possible to construct a one-parameter family of 
 nearly
 Sasakian non-Sasakian manifolds. Thus the theory of nearly Sasakian
 non-Sasakian manifolds is essentially equivalent to
 the one of Sasaki-Einstein manifolds.


 Concerning other
dimensions,  there have been many attempts of finding explicit examples of
proper nearly Sasakian non-Sasakian manifolds until the recent result obtained in \cite{agi} showing
that every nearly Sasakian structure of dimension greater than five is always 
Sasakian.  Such result depends on the early work~\cite{mg} by the
first and third authors, which in turn draws many properties proved in ~\cite{olszak}. 
This makes the proof to be spread over several different texts with different
notation. 

The aim of this note is to provide a complete and streamlined proof of the
aforementioned dimensional restriction on nearly Sasakian non-Sasakian manifolds.  We will
also pinpoint where the positivity of the Riemannian metric is used. For this
purpose we work in the more general setting of pseudo-Riemannian geometry. We
will always assume that the metric is non-degenerate.

\bigskip

This paper was written on occasion of the conference \emph{RIEMain in Contact},
held in Cagliari (Italy), 18--22 June 2018.

\section{Preliminaries}
\subsection{Tensor calculus notation.} 

In this section review the notation for the tensor calculus we use throughout the
paper. 

Given a permutation $\sigma \in \Sigma_q$, we will denote by the same symbol
the $(q,q)$-tensor $TM^{\otimes q}\to TM^{\otimes q} $ defined by $\sigma
(X_1\otimes \dots \otimes X_q) = X_{\sigma^{-1}(1)} \otimes \dots \otimes
X_{\sigma^{-1} (q)}$.

Let $\cov$ be a covariant derivative. 
It is easy to show  that $\cov \sigma =0$. 
If $T$ is an arbitrary $(p,q)$-tensor, then $\cov T$ can be considered as a  $(p,q+1)$-tensor. 
We define recursively the $(p,q+k)$-tensors $\cov^k T$ by $\cov^{k+1} T := \cov
(\cov^k T)$. 

We will 
use the following convention regarding the arguments of $\cov^k T$
\begin{equation*}
(\cov^k T)(X_1\otimes \dots \otimes X_{q+k}) : = (\cov^k_{X_1, \dots,
X_k}T)(X_{k+1}\otimes \dots \otimes
X_{q+k}). 
\end{equation*}
Given $T_1$ and $T_2$ of valencies $(p_1,q_1)$, $(p_2,q_2)$, respectively, and such
that  $q_1\ge p_2$, we define
the tensor $T_1\circ T_2$ of type $(p_1, q_1-p_2+q_2)$ by 
\begin{equation*}
(T_1\circ T_2) (X_1,\dots X_{q_1-p_2}, Y_1,\dots Y_{q_2}) = T_1(X_1,\dots,
X_{q_1-p_2}, T_2(Y_1,\dots , Y_{q_2})).
\end{equation*}

%

Note that with our convention for $\cov T$, if $T_1$ and $T_2$ are tensors of valencies $(p_1,q_1)$ and $(p_2,q_2)$ respectively, then
\begin{equation*}
\label{covtens}
\cov(T_1\otimes T_2) = \cov T_1 \otimes T_2 + (T_1 \otimes \cov T_2) \circ
(q_1+1, \dots, 2,1),
\end{equation*}
where we used the cycle notation for permutations, as we will do throughout the
paper.
Moverover, one has
\begin{equation}
\label{covcomp}
\cov(T_1\circ T_2) = \cov T_1 \circ T_2 + T_1 \circ \cov T_2 \circ
(q_1-p_2+1,\dots, 2,1). 
\end{equation}
Of course if $q_1= p_2$, then we get just $\cov(T_1 \circ T_2) = (\cov
T_1)\circ T_2 + T_1 \circ (\cov T_2)$. 
Suppose $T_2=\sigma$ is a permutation in $\Sigma_{q_1}$. Then~\eqref{covcomp} 
should be used with caution since in the term $\cov T_1\circ \sigma$, we have to
consider $\sigma$ as an element of $\Sigma_{q_1}$, not as an element of
$\Sigma_{q_1 +1}$. Let us denote by $s$ the inclusion $\Sigma_{q_1}$ into
$\Sigma_{q_1+1}$ defined by $s(\sigma)(i) = \sigma(i-1) +1$, $i\ge 2$,
$s\left( \sigma \right) (1)=1$. Then $\cov (T \circ \sigma) = \cov T \circ
s(\sigma)$. In the computations below, we will always substitute $\sigma$ with
$s(\sigma)$ when needed, so that if in the composition chain the tensor
$T$ of type $(p,q)$ is followed by a permutation $\sigma$ then $\sigma$ is
always in
$\Sigma_q$.

\subsection{Nearly Sasakian manifolds}
The definition of Sasakian manifolds was motivated by study of local properties
of K\"{a}hler manifolds. Namely, \emph{Sasakian manifold} is an odd dimensional
Riemannian manifold $(M, g)$ such that the metric cone $(M \times \R_+, t g + dt^2)$ is K\"{a}hler. 
Sasakian manifolds can also be characterized as a subclass of almost contact
metric manifolds. 
\begin{definition}
An~\emph{almost contact metric manifold} is a tuple
$(M^{2n+1},g,\phi,\xi,\eta)$, where 
\begin{enumerate}[1)]
\item $g$ is a Riemannian metric; 
\item $\phi$ is a $(1,1)$-tensor;
\item $\xi$ is a vector field on $M$;
\item $\eta$ is a $1$-form on $M$
\end{enumerate}
such that 
\begin{enumerate}[$i$)]
\item $\phi^2 = -\id + \xi\otimes \eta$
\item  $\eta(X)=g(X,\xi)$, $g(\xi,\xi)=1$;
\item $\phi$ is skew symmetric, i.e. $g\circ (\phi \otimes \id) = - g \circ
(\id \otimes \phi)$. 
\end{enumerate}
\end{definition}
From the definition it follows that $\phi \xi =0$ and $\eta \circ \phi =0$. 

By~\cite[Theorem~6.3]{blairbook} the following can be used as an alternative definition of
Sasakian manifolds. 
\begin{definition}
A~\emph{Sasakian} manifold is an almost contact metric manifold $(M,
g,\phi,\xi,\eta)$ such that 
\begin{equation}
\label{cov_phi}
(\cov_X \phi) Y = g(X,Y) \xi - \eta(Y)X. 
\end{equation}
\end{definition}
Nearly Sasakian manifolds where introduced in~\cite{blair_nearly_sasakian} as a
generalization of Sasakian manifolds by relaxing the
condition~\eqref{cov_phi}. 
\begin{definition}
A~\emph{nearly Sasakian} manifold is an almost contact metric manifold $(M,
g,\phi,\xi,\eta)$ such that 
\begin{equation}
\label{cov_phi_nearly_sasakian}
(\cov_X \phi) X = g(X,X) \xi - \eta(X)X. 
\end{equation}
\end{definition}
By polarizing at $X$ the condition~\eqref{cov_phi_nearly_sasakian} can be
restated in the form 
\begin{equation}
\label{cov_phi_nearly_sasakian_}
(\cov \phi -  \xi\otimes g + \eta\otimes \id) ( 1 + (1,2)) =0. 
\end{equation}

As explained in the introduction, we will work in the more general setting of pseudo-Riemannian geometry. 
The definitions of \emph{nearly pseudo-Sasakian} and 
\emph{pseudo-Sasakian} manifolds are the same as above with only distinction
that now $g$ is a pseudo-Riemannian metric. 

We start with establishing some simple properties of nearly pseudo-Sasakian
manifolds. In the case of nearly Sasakian manifolds they were proved in~\cite{blair_nearly_sasakian}. 
\begin{proposition}
\label{easyFacts}
If $(M, g,\phi,\eta)$ is a nearly pseudo-Sasakian manifold then
\begin{enumerate}[i)]
\item for any vector field $X$, the vector field $\cov_X \xi$ is orthogonal to
$\xi$, equivalently $\eta \circ \cov \xi =0$; 
\item $\cov_\xi \xi =0$ and $\cov_\xi \eta =0$;
\item the operators $\cov_\xi \phi$ and $\phi\circ \cov_\xi \phi$ are skew-symmetric and anticommute with $\phi$; 
\item $\cov_\xi \phi = \phi (\phi +\cov \xi)$ and $\phi + \cov \xi + \phi \circ \cov_\xi \phi =0$. 
\item $(\cov \xi)^2 + \id - \xi \otimes \eta = (\cov_\xi \phi)^2 = (\phi
\cov_\xi \phi)^2$, in
particular, $(\cov \xi)^2$ commutes with $\phi$; 
\item $\xi$ is a Killing vector field or, equivalently, $\cov \xi$ is a
skew-symmetric operator; 
\item $d\eta = 2 \cov \eta = -2 g \circ \cov \xi$. 
\end{enumerate}
\end{proposition}
\begin{proof}
Applying $\cov_X$ to $1= g(\xi,\xi)$, we get 
\begin{equation*}
0= g(\cov_X \xi, \xi) + g(\xi, \cov_X\xi ) = 2 g(\cov_X \xi, \xi) = 2 (\eta
\circ \cov \xi)  (X), 
\end{equation*}
which is equivalent to $\cov_X \xi \perp \xi$. 

To show that $\cov_\xi \xi=0$, we proceed as follows. 
Fist we substitute $X=\xi$ in $(\cov_X \phi)X = g(X,X)\xi -\eta(X)X$ and obtain $(\cov_\xi \phi ) \xi  = 0.$
As $\phi \xi=0$, this implies $\phi (\cov_\xi \xi)=0$.
Therefore 
\begin{equation*}
0 = \phi^2 (\cov_\xi \xi) = -\cov_\xi \xi + \eta(\cov_\xi \xi) \xi.
\end{equation*}
Since $\eta \circ \cov \xi=0$, the above equation implies 
\begin{equation}\label{bsy3.2}
\cov_\xi \xi = \eta(\cov_\xi \xi) \xi =  0,\quad \quad \cov_\xi \eta = \cov_\xi
(g\circ (\xi
\otimes \id)) =  g \circ (\cov_\xi \xi \otimes \id) = 0. 
\end{equation}
To see that $\cov_\xi \phi$ is skew-symmetric it is enough to apply $\cov_\xi$ to the equation
$g \circ (\phi \otimes \id + \id \otimes \phi) =0$. 
To show that $\cov_\xi \phi$ anticommutes with $\phi$ we apply $\cov_\xi$ to
the equation $\phi^2 = -\id + \xi \otimes \eta$ and use $\cov_\xi \xi =0$,
$\cov_\xi \eta =0$. 
Now, that $\phi\cov_\xi \phi$ is skew-symmetric and anticommutes with $\phi$
follows from the following computation
\begin{equation*}
\begin{aligned}
g(\phi  (\cov_\xi \phi) X , Y) &= - g( (\cov_\xi \phi) X, \phi Y) =
g(X, (\cov_\xi \phi)  \phi Y  ) = -g (X, (\phi \cov_\xi \phi)Y) \\[1ex] \phi (\phi
\cov_\xi \phi) & =  -\phi ( (\cov_\xi \phi) \phi) = - (\phi \cov_\xi \phi )
\phi.
\end{aligned}
\end{equation*}
Next we show that $\cov_\xi \phi = \phi (\phi + \cov \xi)$. 
First we polarize $(\cov_X \phi) X = g(X,X) \xi -\eta(X)X$ with respect to
$X$, and get that for any two vector fields $X$ and $Y$
\begin{equation}\label{polarization_cov_phi}
(\cov_X \phi) Y + (\cov_Y \phi)X = 2g(X,Y) \xi - \eta(X) Y -
\eta(Y)X.
\end{equation}
Taking $Y=\xi$ in the above equation, we obtain
\begin{equation}
\label{covxphixi}
(\cov_X \phi) \xi  + (\cov_\xi \phi) X = \eta (X) \xi - X = \phi^2 X.
\end{equation}
As $\phi \xi =0$, we have $(\cov _X \phi)\xi = \cov_X (\phi \xi) - \phi
(\cov_X \xi) = -\phi (\cov_X \xi) = - (\phi \circ \cov \xi) X$. Thus~\eqref{covxphixi} can be rewritten as $\cov_\xi \phi = \phi
(\phi + \cov \xi)$. 
Now since $\eta \circ \phi =0$ and $\eta \circ \cov \xi =0$, we get
\begin{equation*}
\phi\circ \cov_\xi \phi = \phi^2 (\phi + \cov
\xi) = -\phi +   \xi \otimes (\eta \circ \phi)  - \cov \xi  + \xi \otimes (\eta \circ \cov
\xi) = - \phi - \cov \xi. 
\end{equation*}
Next we show that $(\cov_\xi \phi)^2 = (\cov \xi)^2 + \id - \xi \otimes \eta$.
Since $\phi$ anticommutes with $\cov_\xi \phi$, we get 
\begin{equation*}
\begin{aligned}
(\cov \xi)^2 & = ( -\phi - \phi \circ \cov_\xi\phi )^2 
= \phi^2 + \phi^2 \circ \cov_\xi \phi + \phi \circ \cov_\xi \phi \circ \phi +
\phi \circ \cov_\xi \phi \circ \phi \circ \cov_\xi \phi  
\\ & = - \id + \eta \otimes \xi + (\cov_\xi \phi)^2 - (\cov_\xi \phi)^2 \xi \otimes
\eta 
 = -\id + \eta \otimes \xi + (\cov_\xi \phi)^2,
\end{aligned}
\end{equation*}
where in the last step we used $(\cov_\xi \phi)\xi = \cov_\xi (\phi \xi) - \phi
(\cov_\xi \xi) =0$. 
Since $\phi$ anticommutes with $\cov_\xi \phi$, we get
\begin{equation*}
(\phi \cov_\xi \phi)^2 =  - (\cov_\xi \phi)^2 \phi^2  = 
(\cov_\xi \phi)^2 + (\cov_\xi \phi)\xi \otimes \eta = (\cov_\xi \phi)^2,
\end{equation*}
where we used in the last step $\phi \xi =0$ and $\cov_\xi \xi=0$. 

Next we show that $\phi$ commutes with $(\cov \xi)^2$. 
Since $\phi$ anticommutes with $\cov_\xi \phi$,
it commutes with $(\cov_\xi
\phi)^2$. Thus to show that $(\cov \xi)^2$ commutes with $\phi$, we only have
to check that $\phi$ commutes with $\xi \otimes \eta$. But, as we saw,
$\phi \xi =0$ and $\eta \circ \phi =0$. Thus $\phi \circ (\xi \otimes \eta) = 0
= (\xi \otimes \eta) \circ \phi$. 

Next we prove that $\xi$ is a Killing vector field, which, in view of 
\begin{equation*}
\lie_\xi g = g \circ (\cov \xi \otimes \id + \id \otimes \cov \xi), 
\end{equation*}
is equivalent to the claim that $\cov \xi$ is skew-symmetric. 
But $\cov \xi = -\phi - \phi \circ \cov_\xi \phi$ is a sum of two
skew-symmetric operators, and therefore is skew-symmetric. 

Since $\cov \xi$ is skew-symmetric, we get
\begin{equation*}
\begin{aligned}
d\eta (X,Y) &= (\cov_X\eta)(Y) - (\cov_Y\eta )(X) = 
g (Y, \cov_X \xi) - g ( X, \cov_Y \xi) \\&  = - 2 (g\circ \cov \xi)(X,Y)\\ &  = 
2 g (Y,\cov_X \xi) = 2 (\cov_X ( g \circ \xi)) (Y) = 2 (\cov_X \eta)(Y).  
\end{aligned}
\end{equation*}
\end{proof}
Next we establish that the $1$-form $\eta$ of any nearly Sasakian manifold is contact. 
We use in this proposition that the metric $g$ is positively defined, since this
permits to conclude that the square of $g$-skew-symmetric operator has
non-positive spectrum. This is not true for a general pseudo-Riemannian metric. 
\begin{theorem}[{\cite{mg}}]\label{nearly_implies_contact}
Let $(M^{2n+1}, g, \phi, \xi, \eta)$ be a nearly Sasakian manifold.
Then
\begin{enumerate}[$i$)]  
\item the eigenvalues of $(\cov \xi)^2$ are non-positive and $0$ has multiplicity one
in the spectrum of $(\cov \xi)^2$; 
\item the operator $(\cov \xi)$ has rank $2n$; 
\item $\eta$ is a contact form.
 \end{enumerate}
\end{theorem}
\begin{proof}
By Proposition~\ref{easyFacts}, the operator $\cov_\xi\phi$ is skew-symmetric,
and therefore the eigenvalues of $(\cov_\xi \phi)^2 - \id$ are negative. By the
same proposition $(\cov \xi)^2 - \xi \otimes \eta = (\cov_\xi \phi)^2 -\id$. 
This shows that the spectrum of $A := (\cov \xi)^2 - \xi \otimes \eta$ is
negative and $A$ has rank $2n+1$. Since $\rank(\xi \otimes \eta)=1$ and 
for any two operators $\rank(B+C) \le \rank(B) + \rank(C)$, we conclude that
$2n+ 1 = \rank(A) \le \rank( (\cov \xi)^2 ) + 1$, i.e. the rank of $(\cov \xi)^2 $
is at least $2n$. This shows also that multiplicity of $0$ in the spectrum of
$(\cov \xi)^2$ cannot be greater than one. Since $\xi$ is in the kernel of
$\cov \xi$ we get that the spectrum of $(\cov \xi)^2$ contains $0$, it has
multiplicity one, and $(\cov \xi)^2$ has rank $2n$. 
As $\cov \xi$ is
skew-symmetric by Proposition~\ref{easyFacts}, the rank of $\cov \xi$ coincides
with the rank of $(\cov \xi)^2$. 
Therefore $\rank (\cov \xi) = 2n$. Thus at every point of $M$, there exists
an adapted basis of $T_x M$ of the form $\xi$, $X_1$, \dots, $X_n$,
$Y_1$, \dots, $Y_n$, with the property that $\cov_{X_k}\xi =\lambda_k Y_k$ and
$\cov_{Y_k} \xi = - \lambda_k X_k$ for some $\lambda_k >0$. 
Then
\begin{equation*}
(\eta \wedge (d\eta)^n) ( \xi, X_1, Y_1, \dots, X_n, Y_n) =
n!\cdot 2^n \cdot \prod_{k=1}^n  \lambda_k \not=0.  
\end{equation*}
\end{proof}

\section{Curvature properties of nearly Sasakian manifolds} 
In this section we reestablish curvature properties of nearly Sasakian
manifolds obtained by Olszak in~\cite{olszak}. The main consequence of these
properties, used in the rest of the paper, is an explicit formula for $\cov^2
\xi$ in terms of $\cov \xi$. 

We will use the following notation for curvature tensors
\begin{equation*}
\begin{aligned}
& R_{X,Y}  := \cov^2_{X,Y} - \cov^2_{Y,X},\mbox{ i.e. } R = \cov^2 \circ
( 1 - (1, 2))\\[1ex]
& \curv (X,Y,Z,W)  := g \left(\, R_{X,Y}Z\,,\,W\right). 
\end{aligned}
\end{equation*}
In particular $R\xi$ denotes the $(1,2)$-tensor on $M$ given by
$(R\xi)(X,Y)= R_{X,Y} \xi$. Also
\begin{equation*}
\begin{aligned}
(\curv \circ (1,4,3,2) ) (X,Y,Z,W) & = \curv (Y,Z,W,X) = g (R_{Y,Z}W, X ) = g
(X,R_{Y,Z}W) \\&  =
( g \circ R) (X,Y,Z,W), 
\end{aligned}
\end{equation*}
that is
\begin{equation}\label{curvviagR}
\curv \circ (1,4,3,2) = g \circ R. 
\end{equation}
For every covariant tensor $T\in \Gamma(TM^{\otimes k})$ and endomorphism
$\phi$, we define $i_\phi T \in \Gamma (TM^{\otimes k})$ by 
\begin{equation*}
i_\phi T = T \circ ( \phi \otimes \id^{\otimes (k-1)} + \id \otimes \phi
\otimes \id^{\otimes (k-2)} + \dots + \id^{\otimes (k-1)}\otimes \phi). 
\end{equation*}
In the following series of propositions we show that $i_\phi R$
vanishes on every nearly pseudo-Sasakian manifold.
This generalizes the Olszak's result obtained
in~\cite{olszak} for nearly Sasakian manifolds.  
\begin{proposition}
\label{prop:iphicurv}
Let $(M,g)$ be a pseudo-Riemannian manifold and $\phi$ a linear endomorphism of
$TM$. Then the tensor $i_\phi \curv$ has the following symmetries
\begin{equation}\label{iphisymmetries}
(i_\phi \curv) (1 + (1,2)) = 0,\quad (i_\phi \curv) (1 - (1,3)(2,4)) = 0,\quad
(i_\phi \curv) (1 + (1,2,3) + (1,3,2)) = 0 . 
\end{equation}
\end{proposition}
\begin{proof}
Since $\phi\otimes\id^{\otimes 3} + \id \otimes \phi \otimes \id^{\otimes 2} +
\id^{\otimes 2} \otimes \phi \otimes \id + \id^{\otimes 3} \otimes \phi $
commutes with every element of $\Sigma_4$, the result follows from the
corresponding symmetries of the curvature tensor $\curv$. 
\end{proof}

The following proposition lists a well-known property of tensors with certain symmetries (see e.g. \cite[page 198]{KN1}).
\begin{proposition}
\label{symmetries}
Let $M$ be a manifold and $T$ a $(0,4)$-tensor on $M$ such that 
\begin{equation*}
T (1 + (1,2)) = 0,\quad T (1 - (1,3)(2,4)) = 0,\quad
T (1 + (1,2,3) + (1,3,2)) = 0 .
\end{equation*}
If $T(X,Y,X,Y)=0$ for any pair of vector fields $X$, $Y$ then $T=0$. 
\end{proposition}
In the next proposition we relate the tensors $i_\phi \curv$
 and $R\phi$. 
\begin{proposition}
\label{rphi}
Let $(M,g)$ be a pseudo-Riemannian manifold. If  $\phi\colon TM\to TM $ is
skew-symmetric with respect $g$ then $i_\phi \curv = g \circ (R\phi \otimes \id)
(1 + (1,3)(2,4)) $. 
\end{proposition}
\begin{proof}
The result follows from
\begin{equation*}
\begin{aligned}
g ( (R_{X,Y}\phi ) Z,W) & = g ( R_{X,Y}(\phi Z), W) - g (\phi (R_{X,Y}Z) ,W) \\
& = 
\curv(X,Y,\phi Z,W) + \curv (X,Y,Z,\phi W)
\end{aligned}
\end{equation*}
and symmetries of $\curv$. 
\end{proof}
\begin{proposition}
\label{nps_verification}
\!\! If 
$(M,g,\phi,\xi,\eta)$ is a nearly pseudo-Sasakian manifold then ${i_\phi
\curv =0}$.
Equivalently, $g \circ (R\phi \otimes \id) ( 1 + (1,3)(2,4)) =0$. 
\end{proposition}
\begin{proof}
By Proposition~\ref{prop:iphicurv} the tensor $i_\phi \curv$ has the
symmetries which permit to apply Proposition~\ref{symmetries}. Thus it is
enough to show that $(i_\phi \curv)(X,Y,X,Y)=0$ for all $X$, $Y\in
\Gamma(TM)$. By Proposition~\ref{rphi}, we have $i_\phi \curv = g\circ (R\phi \otimes \id) (1 +
(1,3)(2,4))$.
Thus $(i_\phi \curv)(X,Y,X,Y) = 2 g( (R_{X,Y}\phi) X,Y) $. 
By definition $R_{X,Y}\phi = \cov^2_{X,Y}\phi - \cov^2_{Y,X}\phi$. 
Since $\cov^2_{Y,X}\phi$ is a skew-symmetric operator, we get
\begin{equation*}
(i_\phi \curv)(X,Y,X,Y) = -2 ( g ( (\cov^2_{X,Y}\phi)Y,X) + g (
(\cov^2_{Y,X}\phi)X, Y)). 
\end{equation*}
From the above expression it follows that $(i_\phi \curv)(X,Y,X,Y)=0$ if and
only if the form 
$Q(X,Y) := g( (\cov^2_{Y,X}\phi)X,Y)$ satisfies $Q(X,Y) = - Q(Y,X)$. 
In the remaining part of the proof we will show  that $Q(X,Y) = (1/ 2)d\eta
(X,Y) g(X,Y)$. Then the result follows since $d\eta$ is skew-symmetric and
$g$ is symmetric. 

Applying $\cov$ to the defining condition for nearly pseudo-Sasakian structure 
\begin{equation*}
(\cov \phi - \xi\otimes g + \eta \otimes \id) ( 1 + (1,2))=0,
\end{equation*}
 we get 
\begin{equation}
\label{covtwophi}
(\cov^2 \phi - \cov\xi \otimes g+ \cov \eta \otimes \id)( 1 + (2,3)) =0. 
\end{equation}
Substituting $(Y,X,X)$ in~\eqref{covtwophi} and then applying $g(-,Y)$ to the
result, we get
\begin{equation*}
\label{covYXphi}
2 ( Q(X,Y) - g (\cov_Y \xi,Y) g(X,X) + (\cov_Y \eta)(X) g(X,Y) ) = 0.
\end{equation*}
By Proposition~\ref{easyFacts}, $(\cov_Y \eta)(X) = (1 / 2) d\eta (Y,X)$ and
$\cov \xi$ is 
skew-symmetric, which implies that $g(\cov_Y \xi , Y) = 0$. 
Hence 
$ Q(X,Y)= (1 / 2) d\eta (X,Y) g(X,Y)$ 
as promised. 
\end{proof}
\begin{proposition}
\label{curvReeb}
Let $(M,g,\phi,\xi,\eta)$ be a nearly pseudo-Sasakian manifold. Then $\curv\circ
\xi|_{\xi^{\perp}} =0$. 
\end{proposition}
\begin{proof}
Let $X$, $Y$, $Z\in \xi^{\perp}$. 
We evaluate $i_\phi \curv=0$ on the quadruples $( \phi X, Y,
Z,\xi)$, $( X, \phi Y , Z,\xi)$, $( X , Y, \phi Z,\xi)$, and $( \phi X , \phi Y
, \phi Z,\xi)$. As $\phi^2|_{\xi^{\perp}} =-\id$ and, by
Proposition~\ref{easyFacts}, $\phi \xi=0$, this gives
the relations 
\begin{equation*}
\begin{aligned}
- (\curv\circ\xi)( X, Y, Z) & + & (\curv\circ\xi)( \phi X, \phi Y, Z) & + & (\curv\circ\xi)( \phi X, Y ,\phi Z) & =0\\ 
 (\curv\circ\xi)( \phi X, \phi Y, Z) & - & (\curv\circ\xi)( X ,  Y, Z) & + & (\curv\circ\xi)(
X,\phi Y ,\phi Z) & =0\\ 
 (\curv\circ\xi)(\phi X, Y, \phi Z) & + & (\curv\circ\xi)(  X, \phi Y, \phi Z) &- & (\curv\circ\xi)( X, Y ,Z) & =0\\ 
 - (\curv\circ\xi)( X,\phi Y, \phi Z) & - & (\curv\circ\xi)( \phi X,  Y, \phi Z) & - &
(\curv\circ\xi)(\phi  X, \phi Y ,Z) & =0 
\end{aligned}
\end{equation*}
Summing up the first three equations with the last one taken twice, we obtain that 
$-3 (\curv \circ \xi)( X,Y,Z)= 0$, and thus $\curv \circ \xi|_{\xi^{\perp}}
=0$. \end{proof}
\begin{proposition}
\label{RKilling}
Let $(M,g)$ be a pseudo-Riemannian manifold and $\xi$ a Killing vector  field
on $M$. Then $\cov^2 \xi$ can be determined from $R\xi$, namely 
\begin{equation*}
g\circ \cov^2 \xi =  g\circ R\xi \circ (1,2). 
\end{equation*}
\end{proposition}
\begin{proof}
Since $\xi$ is Killing, the operator $\cov \xi$ is skew-symmetric, i.e.
$g \circ ( \cov \xi \otimes \id + \id \otimes \cov \xi) =0$. Applying $\cov$ to
this equation we get $g \circ ( \cov^2 \xi \otimes \id + \id \otimes \cov^2 \xi
\circ (1,2))) =0$. Since $g \circ (\cov^2 \xi \otimes \id) = g \circ \cov^2 \xi
\circ (1,2,3)$, we get 
\begin{equation*}
0 = g\circ \cov^2 \xi \circ ( (1,2,3) + (1,2))= g \circ \cov^2 \xi \circ
( (1,3) + 1)(1,2). 
\end{equation*}
Thus 
\begin{equation}\label{gcovtwoxisym}
g\circ \cov^2 \xi = - g\circ \cov^2 \xi \circ (1,3). 
\end{equation}
Next denote $g\circ \xi$ by $\eta$. 
Since $\xi$ is Killing, by repeating the computation in the last step of the
proof of Proposition~\ref{easyFacts}, we get $d\eta = -2 g \circ \cov \xi$. 
This implies
\begin{equation}
\label{ddeta}
\begin{aligned}
0 = d^2\eta &  = (\cov d\eta) (1 + (1,2,3)+ (1,3,2)) = 
-2 (g\circ \cov^2 \xi\circ (1,2))) (1 + (1,2,3) + (1,3,2)) 
\\& = -2 g \circ \cov^2 \xi \circ (1+ (1,2,3) + (1,3,2)) (1,2). 
\end{aligned}
\end{equation}
Now from~\eqref{gcovtwoxisym} and~\eqref{ddeta}, we get
\begin{equation*}
\begin{aligned}
g \circ R\xi & = g \circ \cov^2 \xi \circ (1 - (2,3)) = 
- g \circ \cov^2 \xi \circ ( (1,3) + (2,3)) \\ & = 
- g \circ \cov^2 \xi \circ (1 + (1,2,3) ) (1,3) 
 =  g \circ \cov^2 \xi \circ (1,3,2) (1,3) = g \circ \cov^2 \xi \circ (1,2). 
\end{aligned}
\end{equation*}
\end{proof}
In the next proposition we collect several partial results on the curvature
tensor of a nearly pseudo-Sasakian manifold. 
\begin{proposition}
\label{covtwoxi}
\label{curvexi}
Let $(M,g,\phi,\xi,\eta)$ be a nearly pseudo-Sasakian manifold. Then 
\begin{equation*}
\begin{aligned}
& R\xi = \eta \wedge (\cov \xi)^2,\quad 
\cov^2 \xi = - (\cov \xi)^2 \otimes \eta +  (g \circ (\cov \xi)^2) \otimes
\xi\\[1ex]
& R_\xi = (\cov \xi)^2 \otimes \eta - \xi\otimes g \circ (\cov \xi)^2 \\[1ex] 
& (R\phi) \xi = -\eta \wedge \phi (\cov \xi)^2,\quad
R_{\xi}\phi = -  (\cov \xi)^2 \phi \otimes \eta - (g\circ \phi(\cov \xi)^2)
\otimes \xi. 
\end{aligned}
\end{equation*}
\end{proposition}
\begin{proof}
From Proposition~\ref{curvReeb}, we know that $\curv(X,Y,Z,\xi)=0$ for any
$X$, $Y$, $Z\in \xi^\perp$. As $\curv$ is skew-symmetric on the last two
arguments, we conclude that $g(R_{X,Y} \xi, Z)=\curv (X,Y,\xi,Z)=0 $. Thus
$R_{X,Y} \xi$ is proportional to $\xi$. Hence $R_{X,Y} \xi = \eta( R_{X,Y} \xi)
\xi = \curv (X,Y,\xi,\xi) \xi =0$ for $X$, $Y\in \xi^\perp$. 
This implies
\begin{equation}
\label{rxyxi}
R_{X,Y} \xi = \eta(X) R_{\xi, Y} \xi - \eta (Y) R_{\xi, X} \xi. 
\end{equation}
Thus it is enough to compute $R_{\xi, X} \xi$ or, equivalently,
$\curv(\xi,X,\xi,Y)$. 
Since $\curv(\xi,X,\xi,Y)$ is symmetric with respect to the swap of $X$ and
$Y$, it suffices to find formula for $\curv(\xi,X,\xi,X)$.
By~Proposition~\ref{easyFacts} the operator $\cov\xi$ is skew-symmetric, and
thus also $\cov^2_\xi \xi$ is skew-symmetric. This implies
\begin{equation*}
\begin{aligned}
\curv(\xi,X,\xi,X) & = g( \cov^2_{\xi,X}\xi,X) - g (\cov^2_{X,\xi}\xi,X) = 0  - g
(\cov_X (\cov_\xi \xi),X) + g ( \cov_{\cov_X \xi}\xi, X) \\ &  = g( (\cov\xi)^2 X,
X). 
\end{aligned}
\end{equation*}
Polarizing at $X$, we get $ \curv (\xi, X,\xi, Y) =  g ( (\cov \xi)^2 X, Y )$. 
Therefore $R_{\xi, X} \xi = (\cov \xi)^2 X$. Now~\eqref{rxyxi} can be written
in the form
\begin{equation*}
R \xi = \eta \wedge (\cov\xi)^2.
\end{equation*}
To compute $\cov^2 \xi$, we use the expression $g\circ \cov^2 \xi = (g \circ
R\xi)
\circ (1,2)$ obtained in Proposition~\ref{RKilling}. 
We get that for any $X$, $Y$, $Z\in \Gamma(TM)$
\begin{equation*}
\begin{aligned}
g(X, \cov^2_{Y,Z}\xi) &= g(Y,R_{X,Z}\xi) = \eta(X) g(Y,  (\cov \xi)^2 Z ) -
\eta(Z) g(Y,(\cov \xi)^2 X) 
\\[1ex] & = g(Y, (\cov \xi)^2 Z ) g(X,\xi)-
 g(X,(\cov \xi)^2 Y) \eta(Z).
\end{aligned}
\end{equation*}
The above formula is equivalent to the formula for $\cov^2 \xi$ in the
statement of the proposition since $g$ is non-degenerate. 

Now let $X$, $Y$, $Z$ be arbitrary vector fields on $M$. Then
\begin{equation*}
\begin{aligned}
g(R_{\xi,X}Y, Z) & = \curv(\xi, X,Y,Z) = \curv (Y,Z, \xi, X) 
\\ & = g (R_{Y,Z} \xi, X) = \eta(Y) g ( (\cov \xi)^2 Z, X) - \eta (Z) g
( (\cov \xi)^2 Y, X). 
\end{aligned}
\end{equation*}
Since $(\cov \xi)^2$ is self-adjoint and $g$ is non-degenerate, we get
\begin{equation*}
R_{\xi,X} Y = \eta(Y)(\cov \xi)^2 X - g(  X, (\cov \xi)^2Y) \xi
\end{equation*}
which is equivalent to the formula in the statement. 

To compute $(R\phi)\xi$ we use the already established formula for $R\xi$  
\begin{equation*}
\begin{aligned}
(R_{X,Y}\phi) \xi  & =  R_{X,Y}(\phi \xi) - \phi (R_{X,Y} \xi) = - (\eta \wedge
\phi (\cov \xi)^2)(X,Y). 
\end{aligned}
\end{equation*}
To find $R_\xi \phi$ we use the symmetry property of $g\circ (R\phi \otimes
\id)$ that was proved in Proposition~\ref{nps_verification}. 
We get 
\begin{equation*}
\begin{aligned}
g ( (R_{\xi,X}\phi)Y, Z) &= - g( (R_{Y,Z}\phi) \xi,X) = 
g ( \eta(Y)\phi (\cov \xi)^2 Z,X ) - g ( \eta(Z) \phi (\cov \xi)^2 Y,X)  
\\[1ex] & =  - g(  (\cov \xi)^2 \phi X, Z) \eta (Y) - g(\xi,Z) g(X, \phi (\cov \xi)^2 Y) .  
\end{aligned}
\end{equation*}
Since $g$ is non-degenerate it is equivalent to $R_\xi \phi = - (\cov \xi)^2
\phi \otimes \eta - \xi \otimes (g \circ \phi (\cov \xi)^2)$. 
\end{proof}
\begin{theorem}\label{Reebcharacteristicpolynomial}
Suppose $(M,g,\phi,\xi,\eta)$ is a nearly pseudo-Sasakian manifold. Then the characteristic
polynomial of $(\cov \xi)^2$ has constant coefficients.  
\end{theorem}
\begin{proof}
Throughout  the proof we use that $\cov \xi$ and $\cov^2_Y \xi$ are
skew-symmetric operators. The first fact was proved in
Proposition~\ref{easyFacts}, and the second is its consequence.

The coefficients of the characteristic polynomial of $(\cov \xi)^2$ are
constant if and only if 
the traces of the operators $(\cov \xi)^{2s}$ for $0\le s\le 2n+1$ are
constant. 
In fact, if at some point $p$ of $M$ the spectrum (over $\C$) of $(\cov \xi)^2$
is $(\lambda_1,\dots, \lambda_{2n+1})$
then the $s$-th coefficient of the characteristic polynomial of $(\cov \xi)^2$
is up to the sign an  
elementary symmetric polynomial 
\begin{equation*}
\begin{aligned}
e_s  = \sum_{j_1 < \dots <j_s}
\lambda_{j_1} \cdot \lambda_{j_2}\dots \lambda_{j_s}
\end{aligned}
\end{equation*}
and the trace 
of $(\cov \xi)^{2s}$ is the power sum symmetric polynomial $p_s = \lambda_1^s +
\dots + \lambda_{2n+1}^s$.  
Now the claim follows from the Newton identities 
\begin{equation*}
\begin{aligned}
e_1 & = p_1, \quad 
se_s & = \sum_{j=1}^s(-1)^{j-1} e_{s - j}  p_j,\ s\ge 2. 
\end{aligned}
\end{equation*}

Next, we show that the traces $\trace((\cov \xi)^{2s})$ are constant functions
for all $s \ge 1$. 
Since $\cov$ commutes with contraction, we get that for any vector field
$Y$ on $M$
\begin{equation*}
Y ( \trace\, (\cov \xi)^{2s}) = \trace ( \cov_Y ( \cov \xi)^{2s})
= \sum_{k+\ell=2s-1} \trace \big(\,\, (\cov \xi)^{k} (\cov^2_Y \xi) (\cov
\xi)^\ell
\,\,\big).  
\end{equation*}
By Proposition~\ref{covtwoxi} we know that $\cov^2 \xi = - (\cov \xi)^2 \otimes
\eta + \xi \otimes (g\circ (\cov \xi)^2)$. Since $\cov_\xi \xi=0$ and
$\eta\circ \cov \xi=0$ by Proposition~\ref{easyFacts}, we get $(\cov \xi)\circ
(\cov^2_Y \xi) \circ \cov \xi =0$. 
Thus
\begin{equation}
\label{Ytrcovxitwos}
\begin{aligned}
Y ( \trace\, (\cov \xi)^{2s}) &  
=  \trace \big(\,\,  (\cov^2_Y \xi) (\cov \xi)^{2s-1} \,\,\big)
  + 
\trace \big(\,\,   (\cov \xi)^{2s-1}(\cov^2_Y \xi) \,\,\big).
\end{aligned}
\end{equation}
Since 
the trace of a nilpotent operator is always zero and
\begin{equation*}
\begin{aligned}
\big(\,(\cov^2_Y \xi) (\cov \xi)^{2s-1} \,\big)^2 &= 
(\cov^2_Y \xi) (\cov \xi)^{2s-1} (\cov^2_Y \xi) (\cov \xi)^{2s-1} =0\\
\big(\, (\cov \xi)^{2s-1}(\cov^2_Y \xi)\,\big)^2 &= 
 (\cov \xi)^{2s-1}(\cov^2_Y \xi)
 (\cov \xi)^{2s-1}(\cov^2_Y \xi) = 0,
\end{aligned}
\end{equation*}
we conclude that the both traces in~\eqref{Ytrcovxitwos} are zero and therefore
$\trace(\cov \xi)^{2s} $ is a constant function for all $s$. 
\end{proof} 
In the case of nearly Sasakian manifolds
Theorem~\ref{Reebcharacteristicpolynomial} implies the existence of a tangent
bundle decomposition into a direct sum of subbundles. This decomposition will
be crucial in our proof of Theorem~\ref{mainish} which gives an explicit
formula for $\cov \phi$ on a nearly Sasakian manifold. 
Recall that by Theorem~\ref{nearly_implies_contact}   the spectrum of
$(\cov\xi)^2$ on a nearly Sasakian manifold is non-positive. 

\begin{proposition}
\label{TMdecomposition}
Let $(M,g,\phi,\xi,\eta)$ be a nearly Sasakian manifold.
Suppose $0 =\lambda_0> -\lambda_1 > \dots > -\lambda_\ell$ are the roots of the
characteristic polynomial of 
$(\cov\xi)^2$. 
 Then  $TM$ can be written
as a direct sum of pair-wise orthogonal subbundles $V_k \subset TM$ such that,
for every $0\le k \le \ell$, the restriction of
$(\cov \xi)^2$ to  $V_k$ equals $-\lambda_k \cdot \id$.    
\end{proposition}
\begin{proof}
By Proposition~\ref{easyFacts} the operator $\cov \xi$ is skew-symmetric, and
therefore $(\cov\xi)^2$ is symmetric. As $g$ is positively defined this implies
that $(\cov \xi)^2$ is diagonalizable. Denote by $a_k$ the multiplicity of
$-\lambda_k$ in the characteristic polynomial of $(\cov\xi)^2$. Then, by
examining the diagonal form of $(\cov\xi)^2$, one can see that
$\rank( (\cov\xi)^2 +\lambda_k \cdot \id ) = 2n+1 - a_k$ and that $TM$ can be
written as a direct sum of the subbundles $V_k = \ker( (\cov \xi)^2 + \lambda_k \cdot
\id)$.
It is a standard fact that these subbundles are mutually orthogonal and clearly
the restriction of $(\cov \xi)^2$ to $V_k$ equals $-\lambda_k \cdot \id$. 
\end{proof}

\section{Covariant derivative of $\phi$}
In this section we derive a rather explicit formula for $\cov_X \phi$ on a
nearly pseudo-Sasakian manifold. 
We achieve this by computing separately $\cov_X \phi$ on subspaces $\left\langle
\xi \right\rangle$, $\im (\cov_\xi \phi)$, and $\im (\cov_\xi \phi)^\perp\cap
\xi^\perp$. Then, we will use the  formula to prove Theorem~\ref{main}.

\begin{proposition}
\label{covxiphi}
Let $(M,g,\phi,\xi,\eta)$ be a nearly pseudo-Sasakian manifold. Then 
\begin{equation*}
\cov^2_\xi \phi = \eta \wedge (\cov_\xi \phi \circ \cov \xi) - \xi \otimes \Big(g
\circ  (\cov_\xi \phi \circ \cov \xi)\Big).
\end{equation*}
\end{proposition}
\begin{proof}
Applying $\cov$ to the defining relation of nearly pseudo-Sasakian structure
$(\cov \phi -  \xi\otimes g + \eta\otimes \id) ( 1 + (1,2)) =0
$ we get
\begin{equation}\label{covtwophii}
(\cov^2 \phi - \cov \xi \otimes g + \cov \eta \otimes \id) ( 1 + (2,3)) =0. 
\end{equation}
Denote $( \cov \xi \otimes g - \cov \eta \otimes \id) (1 + (2,3))$ by $T$.
Then~\eqref{covtwophii} becomes $(\cov^2 \phi) (1+(2,3)) = T$. 
By definition of $R$ we have $(\cov^2 \phi) (1 -(1,2)) = R\phi$. 
We have the following equality in $\R\Sigma_3$
\begin{equation}\label{expressionoftwo}
2\cdot \mathrm{id} = (1 - (1,2)) (1 + (1,2,3) - (1,3,2)) + (1 + (2,3)) (1 - (1,2,3) + (1,3,2)). 
\end{equation}
Therefore
\begin{equation}\label{covtwophiii}
2 \cov^2 \phi =  R\phi(1 + (1,2,3) - (1,3,2)) + T (1 - (1,2,3) + (1,3,2)). 
\end{equation}
Now we substitute $(\xi,X,Y)$ in~\eqref{covtwophiii}
\begin{equation}\label{covtwophixi}
2 (\cov^2_{\xi,X}\phi)Y = (R_{\xi,X}\phi)Y + (R_{Y,\xi}\phi) X - (R_{X,Y}\phi)
\xi + T(\xi,X,Y) - T(Y,\xi,X) + T(X,Y,\xi). 
\end{equation}
By Proposition~\ref{covtwoxi}, we have
$R_{\xi}\phi = -  (\cov \xi)^2 \phi \otimes \eta - (g\circ \phi(\cov \xi)^2)
\otimes \xi$ and $(R\phi) \xi = -\eta \wedge \phi (\cov \xi)^2$. Therefore the $R$-part of~\eqref{covtwophixi}
evaluates to
\begin{equation*}
\begin{aligned}
- g (X, \phi (\cov \xi)^2 Y) \xi & - \eta(Y) (\cov \xi)^2\phi  X 
\\ & + g (Y, \phi (\cov \xi)^2 X) \xi  + \eta (X) (\cov \xi)^2\phi  Y 
\\ & \phantom{+ g (Y, (\cov \xi)^2 \phi X) \xi}\,\,\, +  \eta (X) \phi (\cov \xi)^2 Y - \eta (Y) \phi (\cov \xi)^2 X 
\\[1ex] &= 2 \Big( - g (X, (\cov \xi)^2 \phi Y) \xi - \eta(Y) \phi (\cov \xi)^2 X + \eta
(X) \phi (\cov \xi)^2 Y\Big),
\end{aligned}
\end{equation*}
where we use that $\phi$ and $(\cov \xi)^2$ commute by
Proposition~\ref{easyFacts}. 
Next,
\begin{equation*}
\begin{aligned}
T(\xi,X,Y) &= 0  \\[1ex]
T(Y, \xi, X) & = 2 (\cov \xi)(Y) \eta (X) - (\cov_Y \eta) (\xi) X - (\cov_Y \eta)
(X) \xi \\ &  = - g (X, (\cov \xi) Y )\xi  + 2\eta(X) (\cov \xi)Y \\[1ex] 
T (X,Y,\xi) &= T (X, \xi, Y) = - g(Y, (\cov \xi)X )\xi + 2 \eta(Y) (\cov \xi)X. 
\end{aligned}
\end{equation*}
Thus the $T$-part of the right side of~\eqref{covtwophixi}
is
\begin{equation*}
2 g(X, (\cov \xi)Y ) \xi + 2 \eta(Y) (\cov \xi)X - 2 \eta (X) (\cov \xi)Y.
\end{equation*}
As a result we get
\begin{equation}
\label{covtwoxiphi}
\cov^2_\xi \phi =  \xi \otimes g \circ (\cov \xi) (\id - (\cov \xi) \phi) -
\eta \wedge (\cov \xi) (\id -(\cov \xi) \phi). 
\end{equation}
By Proposition~\ref{easyFacts} the operator $(\cov \xi)^2$ commutes with 
$\phi$, $\eta \circ \cov \xi $ vanishes, and $\phi (\phi + \cov \xi) = \cov_\xi \phi$. 
Therefore
\begin{equation}
\label{covxiid}
\begin{aligned}
(\cov \xi) (\id - (\cov \xi) \phi )
  & = (\id - \phi \cov \xi) \cov \xi 
    =  ( - \phi^2 + \xi \otimes \eta - \phi \cov \xi ) \cov \xi 
\\[1ex] & =  - \phi (\phi + \cov \xi) \cov \xi = - \cov_\xi \phi \circ \cov \xi.
\end{aligned}
\end{equation}
Substituting~\eqref{covxiid} in~\eqref{covtwoxiphi}, we get the claim of the
proposition. 
\end{proof}
Given two tensor fields  $T_1$ and $T_2$ on a manifold $M$ such that both
products $T_1 \circ T_2$ and $T_2\circ T_1$ make sense, we define 
\emph{commutator} and \emph{anticommutator} of $T_1$ and $T_2$ by $\Big[\,T_1,T_2\,\Big]= T_1
\circ T_2 - T_2 \circ T_1 $ and $\big\{\,T_1,T_2\,\big\} = T_1 \circ T_2 + T_2
\circ T_1$, respectively. 
The aim of the next three propositions is to find $(\cov_X \phi)Y$ on a nearly
pseudo-Sasakian manifold in the case $Y$ is in the image of $\cov_\xi \phi$.
For this we compute $(\cov \phi)(\cov_\xi \phi)$. The later tensor can be
written as a half-sum of $\big\{\,\cov \phi,\cov_\xi \phi\,\big\}$ and $\Big[\,\cov \phi,\cov_\xi \phi\,\Big]$. 
 \begin{proposition}
\label{prop:anticommutator}
Let $(M,g,\phi,\xi,\eta)$ be a nearly pseudo-Sasakian manifold. 
Then 
\begin{equation*}
\begin{aligned}
\big\{ \cov \phi, \cov_\xi \phi \big\} &  = 
2 \eta \otimes (\cov_\xi \phi)^2
  - (\cov_\xi \phi)(\id + \cov_\xi \phi) \otimes \eta
  + \xi \otimes \Big( g\circ (\cov_\xi \phi) (\id - \cov_\xi
\phi)\Big). 
\end{aligned}
\end{equation*}

\end{proposition}
\begin{proof}
Recall that by Proposition~\ref{easyFacts} we have $\cov_\xi \xi=0$ and
$\cov_\xi \eta =0$. 
Applying $\cov^2_\xi$ to the almost contact structure condition  $\phi^2 + \id - \xi \otimes \eta=0$ we get
\begin{equation*}
(\cov^2_\xi \phi)\circ \phi + (\cov_\xi \phi) \circ (\cov \phi ) + (\cov \phi)
\circ (\cov_\xi \phi) + \phi (\cov^2_\xi \phi) - (\cov^2_\xi \xi) \otimes \eta
- \xi \otimes g \circ (\cov^2_\xi \xi) = 0.
\end{equation*}
Applying the formula $\cov^2 \xi= - (\cov \xi)^2 \otimes \eta +  (g \circ (\cov \xi)^2) \otimes
\xi$ obtained in Proposition~\ref{covtwoxi},
we get
\begin{equation*}
(\cov^2_{\xi,Y} \xi) =  - (\cov \xi)^2 \xi \cdot \eta(Y) + g(Y, (\cov \xi)^2 \xi)\xi
= 0.  
\end{equation*}
Therefore
\begin{equation}
\label{covxicovphi}
(\cov_\xi \phi) \circ (\cov \phi ) + (\cov \phi) \circ (\cov_\xi \phi)  = 
- (\cov^2_\xi \phi) \circ \phi - \phi \circ (\cov^2_\xi \phi). 
\end{equation}
We showed in Proposition~\ref{covxiphi} that 
\begin{equation*}
\cov^2_\xi \phi = \eta \wedge (\cov_\xi \phi \circ \cov \xi) - \xi \otimes \Big(g
\circ  (\cov_\xi \phi \circ \cov \xi)\Big).
\end{equation*}
Since $\phi \xi =0$ and $\eta\circ \phi =0$, we
conclude
\begin{equation}
\label{covtwoxiphiphi}
\begin{aligned}
\cov^2_\xi \phi \circ \phi & = \eta \otimes (\cov_\xi \phi \circ \cov \xi \circ \phi) - \xi \otimes \Big(g
\circ  (\cov_\xi \phi \circ \cov \xi \circ \phi)\Big) 
\\[1ex] 
\phi \circ \cov^2_\xi \phi & = \eta \wedge (\phi \circ \cov_\xi \phi \circ \cov \xi) .
\end{aligned}
\end{equation}
Next, we use that by Proposition~\ref{easyFacts} the operators $\phi$ and
$(\cov_\xi \phi)$ anticommute, and $\cov \xi = -\phi(\id +
\cov_\xi \phi)$ to get
\begin{equation}
\label{covxiphicovxiphi}
\begin{aligned}
\cov_\xi \phi \circ \cov \xi \circ \phi & = 
- \cov_\xi \phi \circ \phi (\id + \cov_\xi \phi) \circ \phi 
= -\cov_\xi \phi \circ \phi^2 (\id - \cov_\xi \phi) = 
 \cov_\xi \phi ( \id - \cov_\xi \phi)\\
\phi \circ \cov_\xi \phi\circ \cov \xi & = 
- \phi \circ \cov_\xi \phi \circ \phi(\id + \cov_\xi \phi)
= \cov_\xi \phi \circ \phi^2 ( \id + \cov_\xi \phi) 
= - \cov_\xi \phi ( \id + \cov_\xi \phi). 
\end{aligned}
\end{equation}
Combining~\eqref{covxicovphi},~\eqref{covtwoxiphiphi},
and~\eqref{covxiphicovxiphi} we get 
the statement of the proposition. 
\end{proof}
\begin{proposition}
\label{prop:commutator}
Let $(M,g,\phi,\xi,\eta)$ be a nearly pseudo-Sasakian manifold. 
Then 
\begin{equation*}
\begin{aligned}
\Big[\, \cov \phi, \cov_\xi \phi\, \Big] &  = 
  (\cov_\xi \phi)(\id + \cov_\xi \phi) \otimes \eta
 + \xi \otimes \Big( g\circ (\cov_\xi \phi) (\id - \cov_\xi
\phi)\Big). 
\end{aligned}
\end{equation*}

\end{proposition}
\begin{proof}
By Proposition~\ref{easyFacts}, we know that
$\cov_\xi \phi = \phi ( \phi  +   \cov \xi) $. 
Notice that for any three tensors $A$, $B$, and $C$, such that all pair-wise
compositions are defined,  we have
\begin{equation*}
 [A,B\circ C] = (A\circ B + B\circ A)\circ  C - B\circ (A\circ C + C\circ A) =
\left\{ A,B \right\} \circ C - B \circ \left\{ A,C \right\}. 
\end{equation*}
Thus to find the commutator of $\cov \phi$ with $\cov_\xi \phi$, we only have
to compute the anti-commutators of $\cov \phi$ with $\phi$ and $\cov \xi$. 

We start with the anticommutator between $\cov \phi$ and $\phi$. For this
we apply $\cov$ to the almost contact metric condition $\phi^2 = -\id + \xi
\otimes\eta$, which gives
\begin{equation}
\label{covXphiphi}
\big\{ \cov \phi, \phi \big\} = (\cov \phi) \phi + \phi (\cov \phi) =  (\cov \xi) \otimes \eta - 
\xi \otimes (g\circ \cov \xi), 
\end{equation}
where we are using $\cov \eta = - g \circ \cov \xi$ from Proposition~\ref{easyFacts}. 

To find the anticommutator between $\cov \phi$ and $\cov \xi$, we first
compute the anticommutator between $\phi$ and $\cov \xi$ and then apply
$\cov$ to the resulting formula. By Proposition~\ref{easyFacts}, we know that
$\cov \xi = - \phi - \phi \circ \cov_\xi \phi$ and that $\phi$ anticommutes
with $\phi \circ \cov_\xi \phi$. Therefore,
\begin{equation*}
\begin{aligned}
\phi \circ \cov \xi + \cov \xi \circ \phi = - 2\cdot \phi^2 = 2\cdot \id -
2\cdot \xi \otimes
\eta
\end{aligned}
\end{equation*}
 and hence 
\begin{equation}
\label{covXphi}
(\cov \phi) \circ (\cov \xi) + \phi \circ\cov^2 \xi + \cov^2 \xi \circ \phi +
(\cov \xi) \circ (\cov \phi)  = - 2 \cov \xi \otimes \eta + 2 \xi \otimes
(g \circ \cov \xi). 
\end{equation}
By Proposition~\ref{covtwoxi}, we know that
$\cov^2 \xi = - (\cov \xi)^2 \otimes \eta + \xi \otimes (g \circ (\cov \xi)^2)$. Since $\phi \xi =0$ and $\eta \circ \phi = 0$, we get 
\begin{equation}
\label{phicovtwoXxi}
\begin{aligned}
\phi \circ \cov^2 \xi &= - \phi (\cov \xi)^2  \otimes \eta
\\ \cov^2 \xi \circ \phi & = \xi \otimes \big(\, g \circ (\cov \xi)^2\circ
\phi\,\big). 
\end{aligned}
\end{equation}
Combining~\eqref{covXphi} with~\eqref{phicovtwoXxi} and then adding the result
to~\eqref{covXphiphi}, we get
\begin{multline*}
\big\{\, \cov \phi,  \phi + \cov \xi \,\big\} =
(\phi (\cov \xi)^2  - \cov \xi ) \otimes \eta 
+ \xi \otimes \Big(\, g \circ (  \cov \xi - (\cov \xi)^2 \phi ) \,\Big). 
\end{multline*}
Thus
\begin{equation*}
\begin{aligned}
\big[\cov \phi, \phi ( \phi + \cov \xi)  \big] 
  & =  \big\{\, \cov \phi , \phi \,\big\}\circ (\phi + \cov \xi) 
 - \phi \circ \big\{\, \cov \phi , \phi + \cov \xi \,\big\}  
 \\& = 
     - \xi \otimes \Big( g \circ \cov\xi \circ (\phi + \cov \xi) \Big)
- ( \phi^2 (\cov \xi)^2 -  \phi \cov \xi) \otimes \eta. 
\end{aligned}
\end{equation*}
Next we use that $\cov \xi + \phi + \phi\circ \cov_\xi \phi =0$ and $(\cov \xi)^2 =
(\cov_\xi \phi)^2 - \id + \xi \otimes \eta$ established in
Proposition~\ref{easyFacts} to bring the above expression to the form of the
proposition statement
\begin{equation*}
\begin{aligned}
\cov \xi \circ ( \phi + \cov \xi) & =  \phi\circ( \id + \cov_\xi \phi)\circ \phi \circ
(\cov_\xi \phi) = \phi^2 (\cov_\xi \phi) (\id - \cov_\xi \phi)\\ & =
- (\cov_\xi \phi)(\id - \cov_\xi \phi)\\
\phi^2 (\cov \xi)^2 - \phi \cov \xi & = \phi^2 ( (\cov_\xi \phi)^2 - \id + \xi
\otimes \eta)  + \phi^2 ( \id + \cov_\xi \phi) 
\\ & = - (\cov_\xi \phi) ( \id + \cov_\xi \phi). 
\end{aligned}
\end{equation*}
This completes the proof. 
\end{proof}
\begin{proposition}
\label{olszak}
Let $(M,g,\phi,\xi,\eta)$ be a nearly pseudo-Sasakian manifold. Then for any
$Y$ in the image of $\cov_\xi \phi$, the following equation holds
\begin{equation*}
(\cov \phi)\circ Y= \eta\otimes
((\cov_\xi \phi)Y) + \xi \otimes (g\circ ( \id - \cov_\xi \phi) Y) . 
\end{equation*}
\end{proposition}
\begin{proof}
Let $Z$ be such that $(\cov_\xi \phi)Z =Y$.
Since $(\cov_\xi \phi)\xi =0$ we can assume that $\eta(Z)=0$ by replacing
$Z$ with $Z - \eta(Z) \xi$ if necessary. 
By Proposition~\ref{prop:anticommutator}, we get
\begin{equation*}
\begin{aligned}
\big\{ \cov \phi, \cov_\xi \phi \big\}\circ Z &  = 
2 \eta \otimes ((\cov_\xi \phi)^2 Z )
  + \xi \otimes \Big( g\circ (\cov_\xi \phi)\circ (\id - \cov_\xi
\phi)\circ Z\Big)
\\ & = 2 \eta \otimes ((\cov_\xi \phi) Y) + \xi \otimes \Big( 
g \circ (\id - \cov_\xi \phi) \circ Y \Big).
\end{aligned}
\end{equation*}
Next, by Proposition~\ref{prop:commutator}, we have
\begin{equation*}
\begin{aligned}
\Big[\, \cov \phi, \cov_\xi \phi\, \Big]\circ Z &  = 
  \xi \otimes \Big( g\circ (\cov_\xi \phi) \circ (\id - \cov_\xi
\phi)\circ Z\Big) = \xi \otimes \Big( 
g \circ (\id - \cov_\xi \phi) \circ Y \Big). 
\end{aligned}
\end{equation*}
Thus 
\begin{equation*}
\begin{aligned}
(\cov \phi) \circ Y & = (\cov \phi) \circ (\cov_\xi \phi)\circ Z = 
(1 / 2) \Big(\,
\big\{ \cov \phi, \cov_\xi \phi \big\}\circ Z 
+ 
\Big[\, \cov \phi, \cov_\xi \phi\, \Big]\circ Z 
 \,\Big)  
\\ & = 
 \eta \otimes ((\cov_\xi \phi) Y) + \xi \otimes \Big( 
g \circ (\id - \cov_\xi \phi) \circ Y \Big).
\end{aligned}
\end{equation*}
This finishes the proof. 
\end{proof}
In the next proposition we use that $g$ is positively defined to conclude that
$(\cov_\xi \phi)^2 Y = 0$ implies $(\cov_\xi \phi)Y=0$. This can be false for a
general nearly pseudo-Sasakian manifold. 
\begin{proposition}
\label{minogiulia}
Let $(M,g, \phi,\xi)$ be a nearly Sasakian manifold. Then for any $Y\in\Gamma( \ker
( (\cov\xi)^2 +\id))$, one has 
$(\cov \phi)\circ Y = \xi \otimes (g \circ Y)$. 
\end{proposition}
\begin{proof}
Throughout the proof we will use that by Proposition~\ref{covtwoxi}, we have 
\begin{equation}
\label{eq:covtwoxi}
\begin{aligned}
\cov^2 \xi = - (\cov\xi)^2
\otimes \eta + \xi \otimes (g \circ (\cov\xi)^2) .
\end{aligned}
\end{equation}
First we show that $\im(\cov Y) \subset \ker(\phi\cov_\xi \phi)$.  
Since $(\cov \xi)^2 Y = -Y$, we have
\begin{equation*}
\cov Y = - \cov ( (\cov \xi)^2 Y) = - \cov^2 \xi \circ \cov \xi \circ Y - \cov
\xi \circ \cov^2 \xi \circ Y - (\cov \xi)^2 \circ \cov Y. 
\end{equation*}

Since $\eta \circ \cov \xi =0$ and $(\cov \xi)\xi =0$ by
Proposition~\ref{easyFacts}, using~\eqref{eq:covtwoxi}, we get
\begin{equation*}
\begin{aligned}
\cov Y & = - \xi \otimes \big(\, g \circ (\cov \xi)^3 \circ Y\big) 
+\eta(Y) \otimes (\cov\xi)^3 - (\cov \xi)^2 \circ \cov Y. 
\end{aligned}
\end{equation*}
Notice that
\begin{equation*}
\eta(Y) = g(\xi, Y) = - g( \xi ,(\cov \xi)^2 Y) = 0
\end{equation*}
thus, taking into account $(\cov \xi)^2 Y = -Y$, we get
\begin{equation*}
\cov Y = \xi \otimes (g \circ (\cov \xi) Y) - (\cov \xi)^2  \circ\cov Y.  
\end{equation*}
Applying $\cov$ to $0 = \eta(Y) = g\circ (\xi \otimes Y)$, we get 
$g \circ (\cov \xi \otimes Y) + g \circ (\xi \otimes \cov Y)=0$. Since
$\cov \xi$ is skew-symmetric, this implies that $g\circ (\cov \xi)Y =
\eta\circ \cov Y$. 
Thus
\begin{equation*}
\cov Y = (\xi \otimes \eta)\circ(\cov Y) - (\cov \xi)^2 \circ \cov Y.
\end{equation*}
The above equation means that the image of $\cov Y$ is a subset of the kernel of
the operator $(\cov \xi)^2 - \xi \otimes \eta + \id$. By
Proposition~\ref{easyFacts} this operator equals to $(\phi\cov_\xi \phi)^2$. 
Since $\phi\cov_\xi \phi$ is skew-symmetric by the same proposition
 and $g$ is positively defined by assumption, we get
that $\im(\cov Y) \subset \ker (\phi\cov_\xi \phi) = \ker(\phi + \cov \xi)$. 
Thus
$(\phi + \cov \xi) \circ \cov Y=0$. 

Next, we claim that $(\phi\cov_\xi \phi)Y=0$. For this we compute
\begin{equation*}
\begin{aligned}
(\phi\cov_\xi \phi)^2 Y= \left( (\cov\xi)^2 + \id - \xi \otimes \eta \right) Y =
-Y + Y - 0 = 0. 
\end{aligned}
\end{equation*}
Therefore, arguing as before, we have $(\phi + \cov \xi)Y=0$. Applying
$\cov$ to this equation, we get
\begin{equation*}
0 = (\cov \phi + \cov^2 \xi)\circ Y + (\phi + \cov \xi) \circ \cov Y = 
(\cov \phi )\circ Y + \xi \otimes \left( g\circ (\cov \xi)^2 Y \right) =
(\cov \phi) \circ Y - \xi \otimes (g \circ Y) . 
\end{equation*}
This concludes the proof. 
\end{proof}
\begin{theorem}[\cite{agi}]\label{mainish}
On every nearly Sasakian manifold $(M,g,\phi,\xi,\eta)$ 
\begin{equation}
\label{covphi}
\begin{aligned}
(\cov_X\phi ) Y&  = g(X,Y)\xi - \eta(Y) X \\
 & \phantom{=} +\eta(X) (\cov_\xi \phi)Y - \eta(Y) (\cov_\xi \phi)X - g
\left(  X, (\cov_\xi \phi)Y \right) \xi. 
\end{aligned}
\end{equation}
Equivalently 
\begin{equation*}
\cov \phi = \xi \otimes g - \id \otimes \eta + \eta \otimes (\cov_\xi \phi) -
(\cov_\xi \phi) \otimes \eta - \xi \otimes \big(\, g\circ (\cov_\xi \phi)
\,\big). 
\end{equation*}
\end{theorem}
\begin{proof}
By Proposition~\ref{nearly_implies_contact} the spectrum of $(\cov \xi)^2$ is 
non-positive and the multiplicity of $0$ is one. Let $0< \lambda_1 <\dots <
\lambda_\ell$ be such that $(0,-\lambda_1,\dots,-\lambda_\ell)$ is the spectrum
of $(\cov \xi)^2$. 
By Proposition~\ref{TMdecomposition} the vector bundle $TM$ can be
written as a direct orthogonal sum of the subbundles $V_0$, $V_{1}$,\dots,
$V_\ell$ such that $(\cov \xi)^2|_{V_0} = 0$ and $(\cov \xi)^2|_{V_k} = -
\lambda_k \cdot \id$ with positive $\lambda_k$'s. 
Thus every vector field $Y$ on $M$ can be written as a sum 
$\eta(Y) \xi + \sum_{k=1}^{\ell} Y_k$, where $Y_k$ are such that $(\cov \xi)^2
Y_k = -\lambda_k Y_k$ and $\eta(Y_k) =0$. 

Since both sides of~\eqref{covphi} are linear over $C^\infty(M)$ with respect
to  $Y$, we have to check the validity of~\eqref{covphi} only for $\xi$ and $Y_k$'s
such that $(\cov \xi)^2 Y_k = -\lambda_k Y_k$ and~$\eta(Y_k)=0$. 

For $Y=\xi$ the formula~\eqref{covphi} reduces to 
\begin{equation*}
(\cov_X \phi)\xi = \eta(X) \xi - X - (\cov_\xi \phi)X. 
\end{equation*}
We can see that it holds on every nearly Sasakian manifold by substituting 
$(\xi,X)$ into the defining relation $(\cov \phi - \xi \otimes g + \eta \otimes
\id )(1 + (1,2))=0$. 

Now suppose $Y$ is such that $(\cov \xi)^2 Y = -Y$ and $\eta(Y)=0$.
 By
Proposition~\ref{minogiulia} we know that $(\cov_X \phi)Y = g(X,Y) \xi$. 
Next, from the equality 
\begin{equation}
\label{covxia}
\begin{aligned}
(\cov \xi)^2 - \xi \otimes \eta + \id = (\cov_\xi \phi)^2
\end{aligned}
\end{equation}
proved in Proposition~\ref{easyFacts}, we get that $(\cov_\xi \phi)^2 Y=0$.
Since $\cov_\xi \phi$ is skew-symmetric and $g$ is positively defined, we
conclude that $(\cov_\xi \phi)Y =0$. 
Thus evaluating the right side of~\eqref{covphi} we also get $g(X,Y)\xi$.

Now assume $Y$ is such that $\eta(Y)=0$ and $(\cov \xi)^2 Y = -\lambda Y$
with $\lambda\not\in\left\{ 0,1 \right\}$. 
Then from~\eqref{covxia}, we get $(\cov_\xi \phi)^2 Y = (1-\lambda)Y$ and
$(1-\lambda) \not=0$. This shows that $Y$ is in the image of $\cov_\xi\phi$ and
we can apply Proposition~\ref{olszak} to compute $(\cov_X \phi)Y$. 
We get
\begin{equation*}
(\cov_X \phi)Y = \eta(X) (\cov_\xi \phi)Y + g( X, Y - (\cov_\xi \phi)Y )\xi.
\end{equation*}
Since $\eta(Y)=0$ the right side of~\eqref{covphi} evaluates to
the same expression. This concludes the proof. 
\end{proof}
\begin{remark}
\label{iff}
It follows from~\eqref{covphi} that a nearly Sasakian manifold is Sasakian if
and only if $\cov_\xi \phi =0$. In fact, if $\cov_\xi \phi=0$,
then~\eqref{covphi}
implies 
\begin{equation}
\label{sas}
\begin{aligned}
(\cov_X \phi)Y = g(X,Y)\xi - \eta(Y)X,
\end{aligned}
\end{equation}
 which is the defining
condition of Sasakian structures. In the opposite direction, if $M$ is a
Sasakian manifold, then computing $\cov_\xi \phi$ by~\eqref{sas} we get zero. 
\end{remark} 

\begin{proposition}
\label{strange}
Let  $(M,g,\phi,\xi,\eta)$  be a 
nearly Sasakian manifold. Denote $g\circ (\phi \otimes \id)$ by $\Phi$ and $g \circ ( \cov_\xi \phi \otimes \id)$ by
$\Psi$. Then $\Phi$ and $\Psi$ are differential forms and
\begin{equation*}
d\Phi = 3 \eta \wedge \Psi,\quad \eta \wedge d \Psi = 0,\quad d\eta \wedge \Psi =0. 
\end{equation*}
\end{proposition}
\begin{proof}
The operator $\phi$ is skew-symmetric by definition of an almost contact
metric structure, and $\cov_\xi \phi$ is skew-symmetric by
Proposition~\ref{easyFacts}. This implies that both $\Phi$ and $\Psi$ are two
forms. 

By definition of the exterior differential we have
\begin{equation*}
d\Phi = g \circ (\cov \phi \otimes \id) \circ ( 1 + (1,2,3) + (1,3,2)). 
\end{equation*}
By Theorem~\ref{mainish}, we have 
\begin{equation}
\label{covphib}
\cov \phi = \xi \otimes g - \id \otimes \eta + \eta \otimes (\cov_\xi \phi) -
(\cov_\xi \phi) \otimes \eta - \xi \otimes \big(\, g\circ (\cov_\xi \phi)
\,\big). 
\end{equation}
Notice that
\begin{equation*}
\begin{aligned}
g & \circ ( \xi \otimes g \otimes \id)  = g \otimes \eta\\
g & \circ ( -\id \otimes \eta \otimes \id)  = - (g \otimes \eta) (2,3)\\ 
g & \circ (\eta \otimes (\cov_\xi \phi) \otimes \id) = \eta \otimes \Psi\\
g & \circ ( -(\cov_\xi \phi) \otimes \eta \otimes \id )  = - (\eta \otimes \Psi
)(1,2) \\
g & \circ ( -\xi \otimes \big(\, g\circ (\cov_\xi \phi))\otimes \id \,\big) = 
- (\eta \otimes \Psi)(1, 3). 
\end{aligned}
\end{equation*}
Next observe that for every $\sigma \in \left\{ (1,2), (2,3), (1,3) \right\}$
we have $\sigma (1 + (1,2,3) +(1,3,2)) = (1,2) + (2,3) +(1,3)$. 
Hence
\begin{align*}
(g\otimes \eta)(1 - (2,3)) (1 + (1,2,3) + (1,3,2)) = (g\otimes \eta) (1 -(1,2)) (1
+ (1,2,3) + (1,3,2))
\end{align*}
vanishes,
since $g$ is symmetric.
Therefore
\begin{equation*}
\begin{aligned}
d\Phi & = (\eta\otimes \Psi) ( 1 - (1,2) - (1,3)) (1 + (1,2,3) +(1,3,2))
\\ & = (\eta\otimes \Psi) (1 - 2\cdot (2,3)) (1 + (1,2,3) +(1,3,2))
\\ & = 3 (\eta\otimes \Psi) (1 +(1,2,3) +(1,3,2)) = 3 \eta \wedge \Psi, 
\end{aligned}
\end{equation*}
where we used $(\eta \otimes \Psi)(2,3) = - \eta\otimes \Psi$. 
Now $0 = d^2 \Phi = 3 (d\eta \wedge \Psi + \eta\wedge d\Psi) $
implies that $d\eta \wedge \Psi = - \eta \wedge d\Psi$. Thus it is enough to
show only $\eta \wedge d\Psi =0$. 
For this we have to check that for any $X$, $Y$, $Z \in \ker(\eta)$ one has
$d\Psi (X,Y,Z) =0$. 
In fact we will show that $(\cov_X(\cov_\xi \phi))Y$ is proportional to
$\xi$ for any $X$, $Y\in \ker(\eta)$. Then the result will follow from the
definitions of $\Psi$ and the exterior derivative $d$. We have
\begin{equation*}
\begin{aligned}
\cov (\cov_\xi \phi) & = 
  \cov (  (\cov \phi) \circ ( \xi \otimes \id)) 
 = \cov^2 \phi \circ ( \xi \otimes \id) 
+ (\cov \phi) \circ (\cov \xi \otimes \id). 
\end{aligned}
\end{equation*}
Applying~\eqref{covphib}, we get
\begin{equation*}
(\cov \phi)\circ (\cov \xi \otimes \id) = 
\xi \otimes (g \circ (\cov \xi \otimes \id)) - (\cov \xi)\otimes \eta -
(\cov_\xi \phi) (\cov \xi)\otimes \eta - \xi \otimes (g \circ (\cov \xi \otimes
\cov_\xi \phi)). 
\end{equation*}
Evaluating the right side of the above equation on $(X,Y)$ with $Y\in
\ker(\eta)$ we get a vector field proportional to $\xi$. Thus it is left to
show that $(\cov^2_{X,\xi}\phi)Y$ is proportional to $\xi$. We have
\begin{equation*}
(\cov^2_{X,\xi}\phi)Y = -(R_{\xi,X}\phi) Y + (\cov^2_{\xi,X}\phi)Y,
\end{equation*}
and therefore we can use the expressions for $R_{\xi}\phi$ and $\cov^2_{\xi} \phi$
obtained in Proposition~\ref{curvexi} and in Proposition~\ref{covxiphi},
respectively.
Namely, we have $R_{\xi}\phi = -  (\cov \xi)^2 \phi \otimes \eta - (g\circ \phi(\cov \xi)^2)
\otimes \xi$, which implies that $(R_{\xi,X}\phi)Y$ is proportional to $\xi$ for
$Y\in \ker(\eta)$. Further, $\cov^2_\xi \phi = \eta \wedge (\cov_\xi \phi \circ \cov \xi) - \xi \otimes \Big(g
\circ  (\cov_\xi \phi \circ \cov \xi)\Big)$ implies that $(\cov^2_{\xi,X}\phi)Y$ is proportional to
$\xi$ for $X$, $Y\in \ker(\eta)$.  
This concludes the proof. 
\end{proof}
Notice that we did not use $\dim M \ge 7$ in the above proposition. 
\begin{theorem}\label{main}
Let $(M,g,\phi,\xi,\eta)$ be a nearly Sasakian manifold of dimension greater or
equal to $7$. Then $M$ is a Sasakian manifold. 
\end{theorem}
\begin{proof}
In view of Remark~\ref{iff} it is enough to show $\cov_\xi \phi=0$. As
$g$ is non-degenerate this is equivalent to $\Psi=0$. By
Proposition~\ref{nearly_implies_contact} $\eta$ is a contact form on $M$. Therefore $d\eta$ is
a symplectic form on the distribution $\ker(\eta)$. 
The dimension of this distribution is greater than or equal to six. 
Thus the wedge product by $d\eta$ induces an injective map $\bigwedge^2
\ker(\eta) \to \bigwedge^4 \ker(\eta)$. 
By Proposition~\ref{strange} we know that $d\eta \wedge \Psi =0$. 
Therefore the restriction of $\Psi$ to $\bigwedge^2 \ker(\eta)$ is zero. 
It is left to show that $i_\xi \Psi=0$. This follows from the definition of
$\Psi$ and $(\cov_\xi \phi)\xi=0$, which in turn follows from
Proposition~\ref{easyFacts}. 
\end{proof}
\bibliography{proceedings}
\bibliographystyle{amsplain}
\end{document}